\documentclass[11pt]{article}
\usepackage{mathtools,amsthm,amsfonts,amssymb,mathrsfs,epsfig,bbm}
\usepackage{microtype}	
\usepackage[plain,noend]{algorithm2e}

\usepackage[top=3cm, bottom=3cm, left=3cm, right=3cm, heightrounded, marginparwidth=2.5cm, marginparsep=3mm]{geometry}

\parskip        \medskipamount

\usepackage[margin=1cm]{caption}	
\usepackage[usenames,dvipsnames]{xcolor}
\usepackage[colorlinks=true]{hyperref}

\definecolor{mygreen}{rgb}{0.05, 0.576, 0.03} 		
\definecolor{mygray}{gray}{0.9}					
\definecolor{myred}{rgb}{0.768, 0.09, 0.09} 		
\definecolor{deeppink}{RGB}{255,20,147}

\definecolor{deepblue}{RGB}{20,10,130}

\newtheorem{theorem}{Theorem}
\newtheorem{lemma}{Lemma}
\newtheorem{corollary}{Corollary}

\newtheorem*{theorem*}{Theorem}

\newtheorem{remark}{Remark}

\def\N{\mathbb{N}}
\def\Z{\mathbb{Z}}

\def\R{\mathbb{R}}

\def\P{\mathbb{P}}
\def\E{\mathbb{E}}

\def\NN{\mathcal N}

\renewcommand{\phi}{\varphi}
\renewcommand{\epsilon}{\varepsilon}

\newcommand{\ind}[1]{\mathbbm{1}_{\{#1\}}}

\long\def\symbolfootnote[#1]#2{\begingroup
\def\thefootnote{\fnsymbol{footnote}}\footnote[#1]{#2}\endgroup}

\newcommand{\keywords}[1]{ \noindent {\footnotesize
             {\small \em Keywords and phrases.} {\sc #1} } }
\newcommand{\ams}[2]{  \noindent {\footnotesize
             {\small \em AMS {\rm 2010} subject classification.
             {\rm Primary {\sc #1}; secondary {\sc #2}} } } }

\newcounter{para}

\newcommand{\egd}{\overset{(d)}{=}}

\def\fm{\mathfrak m}

\def\WW{\mathcal W}
\def\+.{+ \text{\tiny \raisebox{0.5mm}{$\bullet$}}}

\renewcommand{\bar}[1]{\overline{#1}}

\newcommand{\Addresses}{{
  \bigskip
  \footnotesize

  Sergey Foss,\\
  \textsc{Sobolev Institute of Mathematics, Siberian Branch of the Russian Academy of Sciences, Novosibirsk, Russia} \\
  \textsc{Novosibirsk State University, Novosibirsk, Russia}\\
  \textsc{School of Mathematical and Computer Sciences, Heriot-Watt University, Edinburgh, UK} \\
  \textit{E-mail address}: \texttt{sergueiorfoss25@gmail.com}
  
  \medskip

  Takis Konstantopoulos, \\
  \textsc{Department of Mathematical Sciences, University of Liverpool, Liverpool, UK} \\
  \textit{E-mail address}: \texttt{takiskonst@gmail.com}

\medskip

  Bastien Mallein, \\
   \textsc{Universit\'e Sorbonne Paris Nord, LAGA, UMR 7539, F- 93430, Villetaneuse, France}\\
  \textit{E-mail address}: \texttt{mallein@math.univ-paris13.fr}
  
  \medskip

  Sanjay Ramassamy, \\
  \textsc{Universit\'e Paris-Saclay, CNRS, CEA, Institut de physique th\'eorique, 91191 Gif-sur-Yvette, France}\\
  \textit{E-mail address}: \texttt{sanjay.ramassamy@ipht.fr}

}}

\begin{document}

\title{\Large \bf Estimation of the last passage percolation constant in a charged
complete directed acyclic graph via perfect simulation} 
\author{\sc Sergey Foss, Takis Konstantopoulos, Bastien Mallein\\ and \sc Sanjay Ramassamy}
\date{}
\maketitle

\begin{abstract}
Our object of study is the asymptotic growth of heaviest paths in a charged (weighted with signed weights) complete directed acyclic graph. Edge charges are i.i.d.\ random variables with common distribution $F$ supported  on $[-\infty,1]$ with essential supremum equal to $1$ (a charge of $-\infty$ is understood as the absence of an edge). The asymptotic growth rate is a constant that we denote by $C(F)$. 
Even in the simplest case where $F=p\delta_1 + (1-p)\delta_{-\infty}$, corresponding to the longest path in the Barak-Erd\H{o}s random graph, there is no closed-form expression for this function, but good bounds do exist. In this paper we construct a Markovian particle system that we call ``Max Growth System'' (MGS), and show how it is related to the charged random graph. The MGS is a generalization of the Infinite Bin Model  that has been the object of study of  a number of papers. We then identify a random functional of the process that admits a stationary version and whose expectation equals the unknown constant $C(F)$. Furthermore, we construct an effective perfect simulation algorithm for this functional which produces samples from the random functional.
\end{abstract}


\ams{82M31}{60K15, 60G10, 05C80}

\section{Preliminaries}
A  Barak-Erd\H{o}s random graph is a directed acyclic version of the standard Erd\H{o}s-R\'enyi graph \cite{BE84}. We let $\Z^+$, the set of non-negative integers, serve as the set of vertices. For each pair of vertices $i, j$ with $i < j$, declare $(i,j)$ as an edge directed from $i$ to $j$ with probability $p$, independently from any other pair. Then the maximum length $L_n$ of all paths from vertex $0$ to $n$ satisfies a law of large numbers \cite{NEWM92,FK03}:
 $\lim_{n \to \infty} L_n/n = C(p)$, a.s., where $C(p)$ is a certain deterministic, increasing, analytic function of $0<p\leq1$ \cite{MR16, MR19}. Owing to the fact that such a graph appears as a model in various natural applications, such as in computer systems \cite{GEL86,ISONEW94}, in mathematical ecology \cite{NEWM92,NEWCOH86} and others, information about $C(p)$ has been the subject of a number of papers \cite{NEWM92, FK03, DFK12, FMS14, MR16, MR19, FK18, FKP20}.
 
Longest paths in Barak-Erd\H{o}s graphs may be seen as a special case of the last passage percolation model, which studies the growth rate of the length of the longest path in a directed acyclic graph whose edges and/or vertices are equipped with random weights. By contrast, first passage percolation is a model studying shortest paths between two points in a graph (usually undirected) whose edges and/or vertices are equipped with random weights. Both models have mainly been studied in the case when the graphs are of the form $\mathbb{Z}^d$, see e.g. \cite{DNV22,ADH17} and references therein. First passage percolation has recently been investigated for Barak-Erd\H{o}s graphs in the sparse inhomogeneous setting \cite{MT22}.

We take interest in a generalization of Barak-Erd\H{o}s graphs, considering the last passage percolation problem on a complete directed acyclic graph, in which each edge has a signed charge distributed according to an independent copy of the random variable  $w$  taking values in
 $\R \cup \{-\infty\}$ with a support bounded from above. We denote by $F$ the law of $w$, and let $\{w_{i,j}, 0 \leq i < j\}$, be a collection of i.i.d.\ copies of $w$. If $\pi$ is a path from $i$ to $j$, namely an increasing collection of vertices $(i=i_0, i_1, \ldots, i_\ell=j)$ then its charge is defined as the sum of the charges of its edges: $w(\pi) = w_{i_0,i_1} + \cdots + w_{i_{\ell-1}, i_\ell}$, using the convention that $-\infty + x = -\infty$ for all $x \in \R$. In other words, if a path goes through an edge with charge $-\infty$, then the charge of the path is $-\infty$. We define by convention the charge of a path consisting of a single vertex as $0$.

We are concerned with the quantity
\begin{equation}
  \label{eqn:lppWeight}
  W_n := \sup\{w(\pi):\, \pi \text{ is a path from $0$ to $n$}\},
\end{equation}
the maximum charge of all paths between $0$ and $n$.
Observe that $(W_n,n\geq 0)$ is a superadditive random sequence as direct computations show that for all $n,m \geq 0$,
\[
  W_{n+m} \geq W_{n} + W_{n,n+m},
\]
with $W_{n,n+m} = \sup\{w(\pi):\, \pi \text{ is a path from $n$ to $n+m$}\}$. As $W_{n,n+m}$ is a copy of $W_m$ independent of $W_n$, by Kingman's subadditive ergodic theorem  \cite{Kingman}
we have
\begin{equation}
  \label{eqn:lppConstant}
  \frac{W_n}{n} \to C(F), \quad \text{a.s. as $n\to\infty$}
\end{equation}
where $C(F)$ is a deterministic function of the law $F$. We refer to $C(F)$ as the \emph{last passage percolation constant} of $F$. The objective of the present article is to develop an approach to compute this constant through 
Monte Carlo
methods.

We denote by 
\[
L = \inf\{z \in \R : \P(w > z)  = 0\}
\] 
the \emph{essential supremum} of $F$ (this is the maximal point of the support of the distribution $F$).

Consider momentarily the case $L \leq 0$.
Then edge weights are nonpositive a.s., and, since 
\[
- W_n =- \inf\{ - w(\pi), \pi \text{ is a path from $0$ to $n$}\},
\]
the problem is that of first passage percolation on the complete directed graph.
We claim that $C(F)=0$. Indeed, with $n \ge 2$, considering the 2-edge path $(0,j,n)$,
\[
  0 \geq W_n \ge \sup_{1 \le j \le n-1} (w_{0,j} + w_{j,n})\quad \text{a.s.},
\]
hence
\begin{align*}
\P(W_n \ge 2(L-1)) 
&\geq \P\left( w_{0,j} + w_{j,n} \ge  2 (L-1) \text{ for all } 1 \le j \le n-1\right)\\
  &\geq 1 - (1 - \P(w > L-1)^2)^{n-1}.
\end{align*}
Thus, by the Borel-Cantelli lemma,
$W_n/n \to 0$ in probability, and hence $C(F)=0$, as claimed. 
It is not hard to see that $W_n$ itself  converges weakly to the 
random variable $\max(w,2L)$.

We only consider the case $L>0$ in the rest of the article. In this situation, up to replacing $w$ by $w/L$, we will assume without loss of generality that the essential supremum of $F$ is $1$. Hence, we work under the following assumption for the distribution $F$ of $w$:
\begin{equation}
  \label{eqn:FAssumption}
  \forall \epsilon > 0, \quad F([1-\epsilon,1])>0 \text{ and } F((1,\infty)) = 0.
\end{equation}

The case $F=p\delta_1+(1-p)\delta_{-\infty}$ formally corresponds to a Barak-Erd\H{o}s graph as any edge with charge $-\infty$ can be ignored. Such a graph was studied in \cite{FK03} and a more general version of it in \cite{DFK12}. The constant $C(p)$ mentioned earlier is, with an abuse of notation, the constant $C(p\delta_1+(1-p)\delta_{-\infty})$. We know that $C(p)>0$ for all $p>0$  which means that eventually, any two vertices that are far apart enough are connected
by a path that has charge $>-\infty$. It was shown in \cite{MR16, MR19} that the function $p\to C(p)$ is analytic on $(0,1]$ and a two-term asymptotic expansion was given in the limit $p\to0$ as well as the power series expansion around $p=1$.

The case $F = p \delta_1 + (1-p) \delta_x$, where $x \in(-\infty,1)$, was studied in \cite{FKP20}. For this case, the quantity $C(F)$ was denoted by $C_p(x)$, a differentiable function of $x \in (-\infty,1)\setminus I$, where $I$ is the union of nonpositive rationals and of the inverses $1/n$, $n \geq 2$. Moreover $C(p)$ is the decreasing limit of $C_p(x)$, as $x \to -\infty$. In the special case when $x=0$, it was shown in \cite{Dutta} that $C_p(0)=1/\psi(1-p)$ where $\psi$ is a Ramanujan theta function.

Let $F$ be a distribution on $[-\infty,1]$ with essential supremum $1$. Comparing $F$ with the distribution $p\delta_{1/2}+(1-p)\delta_{-\infty}$ where $p=F([\tfrac{1}{2},1])$, it is not hard to see that $C(F)>0$. The goal of this paper is to \emph{construct} a random variable with expectation $C(F)$ that can be perfectly simulated via an explicit algorithm. Perfect simulation of a functional of a Markov chain in its ``steady-state'' is a technique that, whenever applicable, avoids the bias introduced by  standard MCMC (=Markov Chain Monte Carlo) methods, in which one would approach $C(F)$ by a realization of $W_n/n$ for $n$ large enough. The terminology and algorithm were introduced in \cite{PROWIL96}.

A survey can be found in \cite{KEN05}. Its relation to the so-called backwards-coupling was studied in \cite{FOSTWE98}. It belongs to the broader area of coupling methods for stochastic recursions that may entirely lack the Markovian property \cite{BORFOS92, COMFERFER02,FK03}.

Our perfect simulation algorithm is based on the construction of a particle system, that we call the {\em Max Growth System} (MGS)
associated to the charged complete directed graph. This particle system can be seen as an extension of the Infinite Bin Model (IBM) \cite{FK03,MR16, CR17, MR19} arising in connection to the Barak-Erd\H{o}s graph. We mention {\em en passant} that the IBM is a particle system in discrete time introduced in \cite{FK03} but one which falls in a natural class of similar particle systems, manifestations of which have appeared frequently in the literature, e.g.\ in \cite{AP83}.

In Section~\ref{sec:maxgrowth}, we first define the MGS with charge distribution $F$ and describe some properties of its dynamics. In Section~\ref{sec:stc}, we show that a certain functional of the MGS is a Markov chain that admits a stationary version. In Section~\ref{sec:perfectsimulation}, we pull the random variable mentioned above from the stationary version, show that its expectation is $C(F)$ and describe a perfect simulation algorithm. We conclude by suggesting further directions of research in Section~\ref{sec:further}.

\section{The Max Growth System (MGS)}
\label{sec:maxgrowth}

The Max Growth System is a particle system on $\R\cup\{-\infty\}$ in which at every step a new atom is added to the process. This auxiliary particle system is constructed in such a way that starting from a single particle at position $0$, the $n$th particle in the system will be placed at position $W_n$. The precise connection between the Max Growth System and the last passage percolation model introduced in the previous section is given in Lemma~\ref{lem:MGSLPP}.

\subsection{Deterministic dynamics of the MGS}

We let $\NN$ be the set of locally finite point measures on $\R\cup\{-\infty\}$ with a finite maximal element, namely, measures whose values are nonnegative integers and which are finite on every interval of the form $[x,\infty)$. Another way to define $\mathcal{N}$ is as the set of Radon measures $\nu$ on $\R\cup\{-\infty\}$ such that $x \mapsto \nu([x,\infty))$ is a non-increasing function from $\R$ to $\Z_+$. This will be the state space on which the MGS is defined. Any such measure $\nu\in\NN$ is specified by the nonincreasing sequence $\nu_1 \ge \nu_2 \ge \cdots \geq -\infty$ of the locations of the points (atoms) of $\nu$. This sequence may be finite or infinite. For example, $\nu=2\delta_0 + \delta_{-1.5} + 3 \delta_{-4}$ is equivalently represented by the finite sequence $(0,0, -1.5, -4,-4, -4)$. We shall therefore think of any $\nu \in \NN$ either as a point measure $\nu=\sum_{k\ge 1} \delta_{\nu_k}$  or as a sequence $(\nu_1, \nu_2, \ldots)$. Note that the zero measure $0$ is an element of $\NN$ and corresponds to an empty sequence of points. The total mass $\|\nu\| = \nu(\R\cup\{-\infty\})$ of $\nu$ is the  number of its points (counted with multiplicity). We let $\inf \nu := \nu_{\|\nu\|}$ be the location of the last point of $\nu$ if $\|\nu\|< \infty$. If $\|\nu\|=\infty$, we let $\inf \nu = -\infty$.

Let $w=(w_1, w_2,\ldots)$ be a  sequence of elements of $\R \cup\{-\infty\}$, such that $\sup_{k \ge 1} w_k\leq 1$. Let $\WW$ be the collection of such sequences. Given $\nu$ a non-zero element of $\NN$, define the quantity
\[
\fm(\nu, w) := \sup_{k \ge 1} (\nu_k+w_k), \quad \nu \neq 0, \quad w \in \WW.
\]
Here the supremum is taken either over all $k\ge 1$ if $\|\nu\|=\infty$ or over $1\leq k\leq \|\nu\|$ if $\|\nu\|$ is finite. Observe that $\fm(\delta_0,w) = w_1$ for all $w \in \WW$. 
The map responsible for the dynamics of the MGS is defined by
\[
  \Psi_w \nu := \nu + 
  \delta_{\fm (\nu,w)},
\]
that consists in adding at every step an atom in the process at a position given by $\fm(\nu,w)$.

We will later employ a  ``coupling from the past" technique. 
To this end, it is worth
describing the MGS starting from an arbitrary point in time. Let $(w(t), t \in \Z)$ be a sequence of elements of $\WW$, $\nu$ a point measure in $\NN$ and $T \in \Z$. The MGS starting from $\nu$ at time $T$ is the process $(\nu(t), t \geq T)$ defined recursively by
\[
  \nu(T) = \nu \quad \text{and}\quad \nu(t+1) = \Psi_{w(t+1)} \nu(t), \quad t \geq T.
\]
When $(w_j(t), j \in \Z^+, t \in \Z)$ is i.i.d. with law $F$, we say that $(\nu(t), t \geq 0)$ is an MGS with charge distribution $F$. To simplify notation, for all $s \leq t \in \Z$, we write
\[
  \Psi_{w}^{s,t} = \Psi_{w(t)} \circ \Psi_{w(t-1)} \circ \cdots \circ \Psi_{w(s)},
\]
in which case we have $\nu(t) = \Psi_w^{T+1,t} \nu$ for all $t > T$.

To consider stationary versions of the MGS, we will sometimes need to work with the particle system seen from the rightmost particle. We denote by $\NN_0$ the set of 
$\nu \in \NN$ with $\nu_1 = 0$. 
For $\nu \in \NN$, we define its shift $\sigma\nu$ seen from the front by
\[
  \int f(x) \, \mathrm{d} (\sigma \nu)(x) :=  \int f(x - \nu_1)\,  \mathrm{d} \nu,
\]
for all $\nu \in \NN$ and all positive bounded functions $f$.
Thus $\sigma : \NN \to \NN$ and can be thought of as:
``place the origin at the position of the rightmost atom''.
For example, $\sigma (\delta_a + \delta_b) = \delta_{0} + \delta_{-|a-b|}$. 
Observe that $\sigma$ is a projection of $\NN$ onto $\NN_0$, 
which is consistent with the definition of the MGS as
\[
  \sigma \Psi_w = \sigma \Psi_w \sigma, 
\]
for all sequences $w \in \WW$. It is also worth mentioning that, for all $\nu \in \NN$,  we have
\begin{equation}
  \label{eqn;remSigma}
  \fm(\sigma \nu, w)  = \fm(\nu,w) - \nu_1.
\end{equation}

\subsection{Decoupling properties of the MGS}

The following lemma shows that if there is a large enough gap in between the first and the second atom in the point measure $\nu$, and the sequence of charges satisfies a ``triangular'' property, then the positions of the new particles only depend on a finite number of charges.
\begin{lemma}[Decoupling property]
\label{lem:triangular}
Fix $\ell \in [0,1)$ and a positive integer $n$. Let $T \in \Z$ and $(w(T+t), t \geq 1)$ be a sequence in $\mathcal W$. Let $\nu$ be a point measure in $\NN_0$ such that $\nu_2 \leq -\ell$. We define the sequences
\[
  \nu(t) = \Psi_{w(t)} \nu(t-1) \quad \text{and} \quad \tilde{\nu}(t) = \Psi_{w(t)} \tilde{\nu}(t-1),
\quad t \ge T,
\]
with $\nu(T) = \nu$ and $\tilde{\nu}(T) = \delta_0$. For all $n \in \N$, if
\begin{equation}
  \label{eqn:triangularProperty}
  \bar{w}(T;t) := \max \{ w_1(T+t), \ldots, w_t(T+t)\} \geq 1 - \ell \quad \text{ for all $1 \leq t \leq n$,}
\end{equation}
then $\fm (\nu(T+n-1),w(T+n)) = \fm(\tilde{\nu}(T+n-1),w(T+n))$.
\end{lemma}

\begin{proof}
It suffices to prove this statement for $T=0$. 
We prove, by induction, that
\begin{equation}
  \label{eqn:recursionHyp}
\bar{w}(0;t) \ge 1-\ell \text{ for all } 1 \le t \le n ~\Rightarrow~
  {\nu}(n)_{|\R^+} =  \tilde{\nu}(n)_{|\R^+} ~\text{and} ~\nu(n)(\R^+) = n+1
\end{equation}
Assume first that $n =1$. In this case,
\[
  \fm(\nu,w(1)) = \max\big( w_1(1),\, \max_{j \geq 2}[\nu_j + w_j(1) ]\big).
\]
Since, by assumption, $\nu_j \le \nu_2 \le -\ell$ for all $j \ge 2$,
we have $\nu_j + w_j(n) \le -\ell+1$ for all $j \ge 2$.
To prove \eqref{eqn:recursionHyp} for $n=1$ we must assume that $w_1(1) \ge 1-\ell$.
But then $w_1(1) \ge \max_{j \geq 2}[\nu_j + w_j(1)]$ and so
\[
  \fm(\nu,w(1)) = w_1(1) = \fm(\tilde{\nu}(0),w(1)).
\]
Hence $\nu(1) = \nu + \delta_{w_1(1)}$ and, with $\tilde \nu(0)=\delta_0$,
$\tilde\nu(1) = \delta_0 + \delta_{w_1(1)}$.
Hence \eqref{eqn:recursionHyp} holds for $n = 1$.

Assume next that \eqref{eqn:recursionHyp} holds for some $n \ge 2$.
We prove that it also holds for $n+1$.
To do this, it suffices to assume that
$\nu(n)_{|\R^+} =  \tilde{\nu}(n)_{|\R^+}$, $\nu(n)(\R^+) = n+1$,
and $\overline w(0; t) \ge 1-\ell$ for all $1\le t \le n+1$.
In this case, we have
\[
  \fm(\nu(n),w(n+1)) = \max\left( \max_{j \leq n+1}[\nu_j(n) + w_j(n+1)],
\,  \max_{j \geq n+2} [\nu_j(n) + w_j(n+1)] \right).
\]
But, for all $j \le n+1$, $\nu_j(n) \ge 0$ and so
\[
\max_{j \leq n+1}[\nu_j(n) + w_j(n+1)] \ge 1-\ell.
\]
Taking into account the assumption 
$\nu_2 \le -\ell$, we have, for all $j \geq n+2$, $\nu_j(n) \leq -\ell$ and so
\[
  \max_{ j \leq n+1}[\nu_j(n) + w_j(n+1)] \geq 1 - \ell \geq 
\max_{ j \geq n+2}[\nu_j(n) + w_j(n+1)],
\]
which implies that 
\[
\fm(\nu(n),w(n+1)) = \fm(\tilde{\nu}(n),w(n+1))
= \max_{j \leq n+1}[\nu_j(n) + w_j(n+1)] \ge 1-\ell > 0.
\] 
The configuration $\nu(n+1)$ is thus obtained by adding a particle to $\nu(n)$ at a 
positive location. Since $\nu(n)=\tilde \nu(n)$ on $\R^+$ and since the particle
is added at the same location for both, we have $\nu(n+1)=\tilde \nu(n+1)$ on $\R^+$.
Clearly, $\nu(n+1)(\R^+)= \nu(n)(\R^+)+1 = n+2$,
so \eqref{eqn:recursionHyp} holds for $n+1$.
\end{proof}

The above lemma allows us to describe a set of conditions on the sequences $(w(t))$ so that the increments of $\nu$ and $\tilde{\nu}$ are algebraically independent of $\nu$ and $\tilde{\nu}$.
\begin{corollary}
\label{cor:good}
Let $T \in \Z$, $n \in \N$, $\ell\in[0,1)$ and $(w(T+t),0 \leq t \leq n)$ a sequence such that
\begin{equation}
  \label{eqn:wCond}
  w_1(T) \geq \ell \text{ and } \min\{\bar{w}(T;1), \bar{w}(T;2) , \ldots , \bar{w}(T;n)\} \geq 1-\ell.
\end{equation}
Let $\nu, \tilde{\nu}$ be two elements of $\NN_0$ and define the sequences
\[
  \nu(t) = \Psi_{w(t)} \nu(t-1) \text{ and } \tilde{\nu}(t) = \Psi_{w(t)} \tilde{\nu}(t-1), \quad t \ge T,
\]
with $\nu(T-1) = \nu$ and $\tilde{\nu}(T-1) = \tilde{\nu}$. Then $\fm(\sigma \nu(t-1),w(t)) = \fm(\sigma \tilde{\nu}(t-1),w(t))$ for all $T+1 \leq t \leq T+n$.
\end{corollary}

In other words, the sequence 
$\left(\fm(\sigma \nu(t-1),w(t)),~ T+1 \leq t \leq T+n\right)$ is algebraically independent of $\nu(T-1)$ provided that $w$ satisfies \eqref{eqn:wCond}.

\begin{proof}
We observe that as $w_1(T) \geq \ell$, we have
\[
  \fm(\nu(T-1),w(T)) \geq \nu_1(T-1) + w_1(T) \geq \ell \quad \text{and} \quad \fm(\tilde{\nu}(T-1),w(T)) \geq \ell.
\]
Therefore, the second largest atoms of $\sigma \Psi_{w(T)} \nu(T-1)$ and $\sigma \Psi_{w(T)} \tilde{\nu}(T-1)$ are both smaller than $-\ell$, hence by \eqref{eqn;remSigma} we can apply Lemma~\ref{lem:triangular}, which completes the proof.
\end{proof}

\subsection{The MGS derived from the charged complete directed graph}
\label{subsec:MGSLPP}

Consider the charged complete directed graph with i.i.d.\ edge charges $\{w_{i,j}, 0 \leq i < j\}$ of law $F$, a collection of i.i.d.\ random variables in $\R \cup \{-\infty\}$ with common law $F$ 
satisfying assumption \eqref{eqn:FAssumption}.
For all $n \in \N$, we write $W_n$ for the length of the longest path
between $0$ and $n$. We observe that $(W_n,n\geq 1)$ can be coupled with the MGS with charge distribution $F$.

\begin{lemma}
\label{lem:MGSLPP}
Let $(\nu(t), t \geq0)$ be an MGS with charge distribution $F$ such that $\nu(0)=\delta_0$ and let $(W_t, t \geq 0)$ as defined in \eqref{eqn:lppWeight}. We have the following equality in distribution:
\begin{equation}
  \label{eqn:aim}
  (\nu(t), t \geq 0) \egd \left(\sum_{j=0}^t \delta_{W_j}, t \geq 0\right).
\end{equation}
\end{lemma}

\begin{proof}
By definition, we have $\nu(0) = \delta_0 = \delta_{W_0}$, using that the path of length $0$ between $1$ and $1$ has mass $0$. Let $t_0 \in \Z^+$ and assume that we can construct a coupling between $\nu$ and $W$ such that $(\nu(t), t \leq t_0) = (\sum_{j=0}^t \delta_{W_j}, t \leq t_0)$ a.s. Conditionally on this coupling, let $(w_{j,t_0+1}, j \in \Z^+)$ and $(w_j(t_0+1), j \in \Z^+)$ be independent families of i.i.d. random variables with law $F$. By \eqref{eqn:lppWeight}, decomposing all paths $\pi$ ending at $t_0+1$ according to their last step, we have
\begin{multline*}
  W_{t_0+1} = \max\{ W_j + w_{j,t_0+1}, 0 \leq j \leq t_0 \}\\
   \egd \max \{ \nu_{j}(t_0) + w_{j}(t_0+1), 1 \leq j \leq t_0+1\} = \fm(\nu(t_0),w(t_0+1)),
\end{multline*}
therefore $\nu(t_0+1) \egd \sum_{j=0}^{t_0+1} \delta_{W_j}$. As a result, we can couple the two sequences of random variables in such a way that the above equality holds almost surely.

Hence, by recursion, there exists a coupling between the MGS and the last passage percolation problem such that \eqref{eqn:aim} holds for all times.
\end{proof}

A noteworthy observation is that the increments of $W$ are the same as the relative increments of the MGS. More precisely, defining
\begin{equation}
  \label{lppMaxWeight}
  M_n = \max_{0 \leq k \leq n} W_k = \sup_{0 \leq k \leq n}\left\{ w(\pi), \pi \text{ path from $0$ to $k$} \right\},
\end{equation}
the increments of the sequence $(M_n, n \geq 0)$ can be connected to the relative increments of the MGS.
\begin{corollary}
\label{cor:Id} 
Under the foregoing assumptions,
\[
(M_n-M_{n-1} ,  n \ge 1) \egd (\fm(\sigma \nu(n-1), w(n))^+, n \ge 1).
\]
\end{corollary}

\begin{proof}
In the proof above, we established a coupling between $(\nu(n), n \geq 0)$ and $(W_k, k \geq 0)$. Under this coupling, for $n \in\Z^+$, we have
\[
  W_n = \fm(\nu(n-1),w(n)) \quad \text{and} \quad M_n = \nu_1(n).
\]
As a result, under this coupling, we have
\[
  M_n - M_{n-1} = (W_n - M_{n-1})^+ = (\fm(\nu(n-1), w(n)) - \nu_1(n-1))^+ = \fm(\sigma \nu(n-1), w(n))^+. \qedhere
\]
\end{proof}

\section{Stationarity via coupling}\label{sec:stc}

We recall that our aim is to compute the quantity $C(F)$ defined by
\[
  C(F) := \lim_{n \to \infty} \frac{W_n}{n} \quad \text{a.s.}
\]
As $F$ has a finite essential supremum, it holds that $\int_0^\infty x F(\mathrm{d}x) < \infty$. Therefore, by \cite{FMS14},  it is known that
\[
  C(F) = \lim_{n \to \infty} \frac{M_n}{n} \quad \text{a.s. and in } L^1.
\]
Thus, if $F$ satisfies \eqref{eqn:FAssumption}, we have
\[
  C(F) = \lim_{n \to \infty} \frac{\E(M_n)}{n}.
\]

Using Corollary \ref{cor:Id}, we remark that for all $n \in \N$,
\[
  \frac{M_n}{n} = \frac{1}{n} \sum_{j=1}^n (M_j - M_{j-1}) = \frac{1}{n} \sum_{j=1}^n \fm(\sigma \nu(j-1),w(j))^+.
\]
We show in this section that $\left( \fm(\sigma \nu(n-1),w(n)), n \geq 1 \right)$ admits a stationary version (where we recall that $\sigma \nu$ is the point measure shifted so that its rightmost element is at position $0$). 
Since the process $\left( \fm(\sigma \nu(n-1),w(n)), n \geq 1 \right)$
is not Markovian, the term
``stationary version'' should be used with caution. For us, it means that it couples with
a stationary process in finite time,  as in the statement of Theorem \ref{thm1} below.
Then, letting $\bar{\fm}$ be the limit in distribution of $\fm(\sigma \nu(n-1),w(n))$ as $n \to \infty$, we have
\[
  C(F) = \lim_{n \to \infty} \frac{1}{n}\sum_{j=1}^n \E\left(\fm(\sigma \nu(j-1), w(j))^+\right) =  \E(\bar{\fm}^+),
\]
as $\fm(\sigma \nu, w)^+ \in [0,1]$ a.s. In the next section, we introduce the perfect simulation algorithm, which consists in giving a realization of $\bar{\fm}$ without constructing the limit in distribution of $\sigma \nu(n)$ as $n \to \infty$. In the special case of Barak-Erd\H{o}s graphs, a simpler case
of a perfect simulation algorithm was explained in \cite{FK03}.

\begin{theorem}
\label{thm1}
Suppose that $F$ is a distribution satisfying \eqref{eqn:FAssumption}. Let $\ell \in [0,1)$ be such that $p :=F([1-\ell,1]) \in (0,1]$. Let $w = (w_i, i \in \N)$ be i.i.d. random variables with law $F$ and $(w(t), t \in \Z)$ i.i.d. copies of $w$. Given $\nu(0) \in \NN$, we define the MGS by
\[
  \nu(t+1) = \Phi_{w(t+1)} \nu(t), \quad \text{ for all $t \geq 0$}.
\]
There exists a stationary process $(\bar{\fm}(t), t \in \Z)$ such that
\[
  \fm(\sigma\nu(t-1),w(t)) = \bar{\fm}(t) \quad \text{a.s. for $t$ large enough}.
\]
In particular $\E( \bar{\fm}(0)^+) = C(F)$.
\end{theorem}

\begin{proof}
For $T \in \Z$ and $t \in \N$, we recall the notation $\bar{w}(T;t) = \max \{ w_1(T+t),\ldots, w_t(T+t)\}$ from Lemma~\ref{lem:triangular}. We introduce the event
\[
  R_k := \bigcap\limits_{j=1}^\infty \{ w_1(k) \geq \ell,  \bar{w}(k;j) \geq 1 - \ell\}.
\]
It is clear from its definition that $(R_k,  k \in \Z)$ is 
a stationary sequence of events with
\[
  \P(R_k) = \P(R_0) = F([\ell,1])\prod_{j=1}^\infty (1 - (1 - p)^j) > 0.
\]
Consider the stationary random set $J := \{ k \in \Z : R_k \text{ holds}\}$. Since $\P(R_k)> 0$, we have, by ergodicity (more specifically by the Poincar\'e recurrence theorem), 
$\inf J = -\infty$ and $\sup J = \infty$ a.s. We enumerate the elements of $J$ by
\[
  \cdots < T_{-1} < T_0 \leq 0 < T_1 < T_2 < \cdots 
\]
We define
\[
  \tilde{\nu}(t) := \sigma \sum_{i \in \Z} \mathbbm{1}_{\{T_i < t \leq T_{i+1}\}} \Psi^{T_i,t}_{w} \delta_0, \quad t \in \Z.
\]
It is clear from its definition that $(\tilde{\nu}(t), t \in \Z)$ is stationary, as $(w(t), t \in \Z)$ is a stationary sequence, and $(\tilde{\nu}(T_i+1), i \in \Z)$ are i.i.d.\ elements of $\NN_0$. Next, we define
\[
  \bar{\fm}(t) = \fm(\tilde{\nu}(t-1),w(t)), \quad t \in \Z,
\]
which is again a stationary sequence.

By Corollary \ref{cor:good}, we observe that for all $t \geq T_1+1$, the quantity $\fm(\sigma \nu(t-1),w(t))$ does not algebraically depend on $\nu(T_1-1)$. Hence, we have $\fm(\sigma \nu(t-1),w(t)) = \bar{\fm}(t)$, using that $\bar{\fm}(t)$ is the same quantity for the MGS started from $\delta_0$ at time $T_1 -1$. As $T_1<\infty$ a.s.\ this completes the proof of the first part of the theorem.

Next, using that
\[
   C(F) = \lim_{n \to \infty} \frac{\E(M_n)}{n} = \lim_{n \to \infty} \frac{1}{n} \sum_{j=1}^n \E(\fm(\sigma \nu(j-1),w(j))^+) \quad \text{a.s.},
\]
and using the eventual equality between $\fm(\sigma \nu(t-1),w(t))$ and $\bar{\fm}(t)$, we have
\[
  C(F) = \lim_{n \to \infty} \frac{1}{n} \sum_{j=1}^n \E(\bar{\fm}(j)^+) = \E(\bar{\fm}(0)^+),
\]
by stationarity and ergodicity of the sequence.
\end{proof}

\begin{remark}
The random times $T_i$ split the process into independent and identically distributed pieces (thereby making the process strictly regenerative) yielding a number of limiting results including a (functional) central limit theorem. In terms of the last passage percolation model, the $T_i$ are the locations of points through which every longest path must pass. Thus the stationary last passage percolation model admits bi-infinite longest paths, and any longest path in a finite graph grown from a single initial vertex will eventually coalesce with some bi-infinite longest path.
\end{remark}

\section{Perfect simulation}
\label{sec:perfectsimulation}

The formula \eqref{eqn:lppConstant} for $C(F)$ suggests a straightforward method for estimating $C(F)$: starting from $\nu(0)=0$, generate iteratively 
$\nu(1), \nu(2), \ldots, \nu(n)$, and take $\nu_1(n)$ for an estimation of $C(F)$. This standard (so-called MCMC) method introduces a bias. Indeed, $\E(\nu_1(n))/n$ is not equal to $C(F)$, but merely converges to that constant.

To eliminate this bias, we produce an algorithm that constructs the variable $\bar{\fm}(0)$, whose distribution is unknown. Then, by standard  
Monte Carlo method, an unbiased estimation of $C(F)$ can be constructed. This is done in this case by using the construction described in the proof of Theorem~\ref{thm1}.

This algorithm is a development of a similar construction for functionals of stochastic recursions in \cite{FK03} and is based on the ideas of so-called ``backward coupling'', see \cite{FOSTWE98}. It is close in spirit to the coupling-from-the-past method for Markov chains \cite{PROWIL96} and to the perfect simulation construction for processes with ``long memory'' \cite{COMFERFER02}. Note that the algorithm from \cite{PROWIL96} is applicable to either finite Markov chains or ordered monotone Markov chains possessing a unique minimal state and a unique maximal state, so 
it cannot be applied to our case.

\begin{theorem}[Perfect simulation]
Define
\[
  T^* := \sup\{ t \le -1:\, {w}_1(t) \geq \ell, \min_{1\leq j \leq |t|} \overline w(t;j) \geq 1-\ell\}.
\]
Then $|T^*|< \infty$ a.s., and
\[
  \bar{\fm}(0) = \fm\left( \sigma \Psi^{T^*,-1}_w \delta_0;w(0) \right) \text{ a.s.}
\]
\end{theorem}

\begin{proof}
We recall that $(T_{-j}, j \in \N)$ are the negative elements of the random set $J$, with $T_{-1} > -\infty$. We remark that
\[
  w_1(T_{-1}) \geq \ell, \quad \bar{w}(T_{-1};j) \geq 1 - \ell \text{ for all $j> 0$},
\]
therefore $T_* \geq T_{-1}$, proving its finiteness.

Moreover, since
\[
  w_1(T^*) \geq \ell, \min_{1 \leq j \leq |t|} \bar{w}(T^*;j) \geq 1-\ell,
\]
by Corollary \ref{cor:good}, the quantity $\fm\left( \sigma \Psi^{T^*,-1}_w \nu; w(0)\right)$ does not algebraically depend on the value of $\nu \in \NN_0$. As a result, it is equal to $\bar{\fm}(0)$, defined as $\fm\left( \sigma \Psi^{T^*,-1}_w \delta_0; w(0)\right)$.
\end{proof}

\begin{remark}
If the essential supremum $L$ of $F$ is infinite, then the perfect simulation algorithm we defined cannot apply. Indeed, in this situation, Lemma~\ref{lem:triangular} does not apply and we could not find an event depending on a finite number of charges such that an analogue of this lemma would hold. When $L=\infty$, even if the tail of $F$ decays fast enough, we would still need to look at infinitely many values of $w_j(1)$ to increment just the first time step of the MGS, making it impossible to hope for a perfect simulation algorithm which ends in finite time for any starting configuration.
\end{remark}

\begin{center}
\begin{algorithm}[ht]
  \caption{Construction of a variable of law $\bar{\fm}(0)$.}
  \label{algo:main}
  \SetAlgoLined
  Fix $t=0$ and $J=1$\;
  Generate the variable $w_1(0)$\;
  Fix Stopping = False\;
  \While{Stopping = False}{
  \While{$\max_{1\le j \leq J} w_j(t) < 1 - \ell$}{Increase $J$ by $1$\;Generate the variable $w_J(t)$\;}
  \While{$J > 1$}{
  Decrease $J$ by $1$ and $t$ by $1$\;
  Generate $w_1(t), \ldots w_J(t)$\;
  \While{$\max_{1\le j \leq J} w_j(t) < 1 - \ell$}{Increase $J$ by $1$\;Generate the variable $w_J(t)$\;}}
  Decrease $t$ by $1$\;
  Generate $w_1(t)$\;
  Fix Stopping = $\{w_1(t) \geq \ell\}$\;
  }
  Fix $\nu = \delta_0$\;
  \For{$s$ from $t+1$ to $-1$}{Generate the variables $w_1(s),\ldots,w_{||\nu||}(s)$ \; Set $\fm = \max \{ \nu_j + w_j(s) \text{ for $1 \leq j\leq\Vert\nu\Vert$}\}$\; Add $\delta_\fm$ to $\nu$\;}
  Set $\fm = \max \{ \nu_j + w_j(0) \text{ for }1 \leq j\leq\Vert\nu\Vert\}$\;
  \textbf{Return:} $\fm - \nu_1$\;
\end{algorithm}
\end{center}

\subsection*{The perfect simulation algorithm}

We now describe more precisely the perfect simulation algorithm. Let $F$ be a probability distribution satisfying \eqref{eqn:FAssumption}, we fix $\ell \in [0,1)$ such that $F([1-\ell,1]) \in (0,1)$. The algorithm requires the construction of an array of i.i.d. random variables with common distribution $F$ until the random variable $T^*$ can be constructed.

To construct $T^*$ as well as $\bar{\fm}(0)$ from the sequence $\{w_j(t), j \in \N, t \in \Z\}$, one only needs to consider a.s. finitely many elements of this set, as $\{T^* = t\}$ is a measurable function of
\[
\{ w_1(t)\}\cup\{w_j(t+k), 1 \leq j \leq k \leq |t| \}
\] 
and $\bar{\fm}(0)$ is a measurable function of
\[
\{ w_1(T^*)\}\cup\{ w_{j}(T^*+k), 1 \leq j \leq k \leq |T^*| \}.
\]
Therefore, we can explore triangular arrays of the form
\[
  \{ w_1(t)\}\cup  \{w_j(t+k), 1 \leq j \leq k \leq |t|\},
\]
progressively decreasing $t$ until time $T^*$ is detected. Once this random variable is known, we construct the random variable $\bar{\fm}(0)$ using the procedure described in Theorem \ref{thm1} from the previously discovered random variables. A possible implementation is described in Algorithm~\ref{algo:main}. We show a graphical representation of a run of Algorithm~\ref{algo:main} in Figure \ref{fig:triangulararray}.

\begin{figure}[ht]
\centering
\includegraphics[width=3.5in]{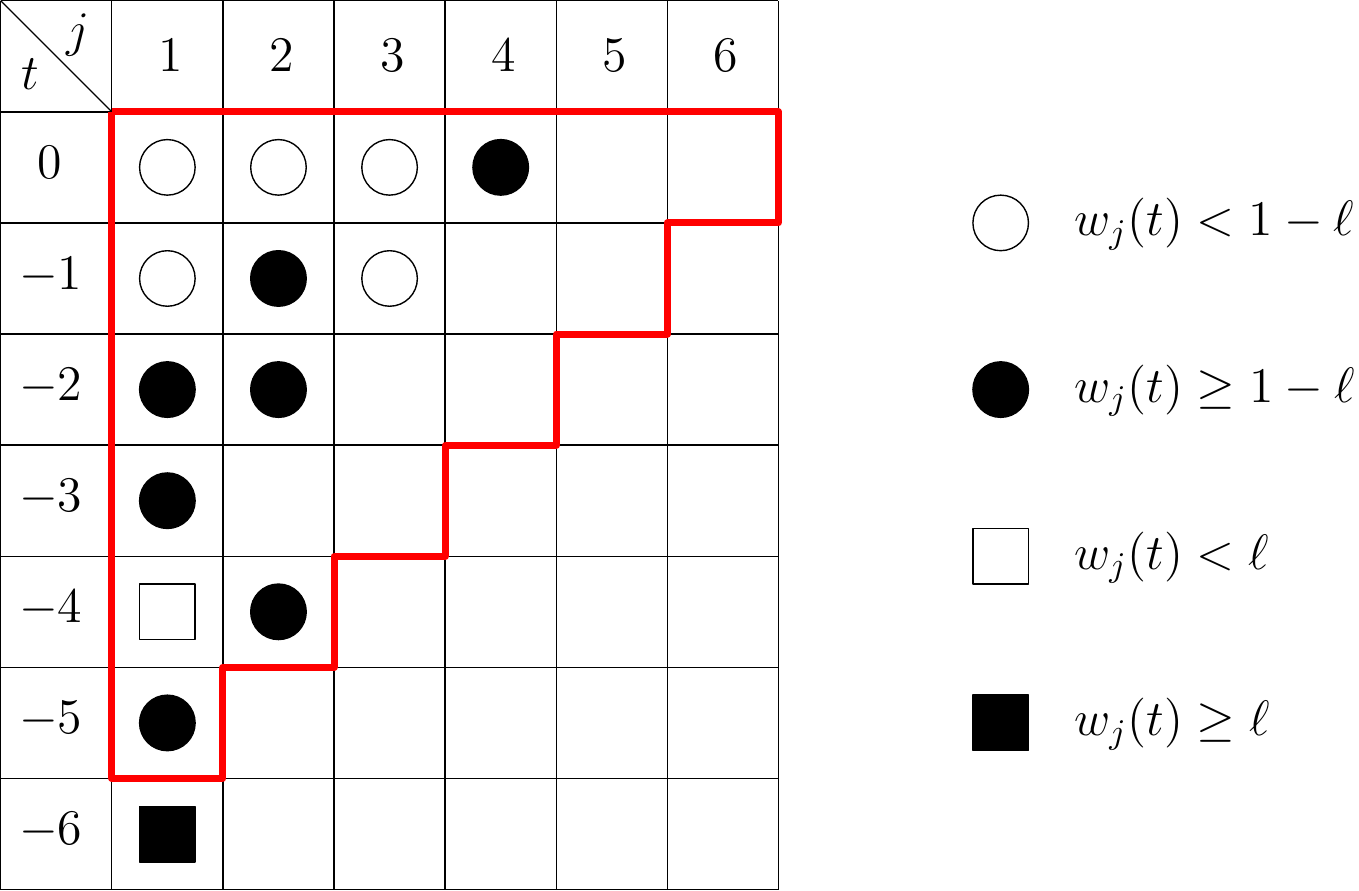}
\caption{Illustration of the execution of Algorithm \ref{algo:main} on an example, in the case where $\ell<1-\ell$. The variables sampled until the Boolean variable Stopping becomes True are pictured by black/white squares and disks. One searches for the first time $T^*$ such that every line of index $T^*+1\leq t \leq 0$ has at least one black disk between columns $1$ and $t-T^*$ and such that there is a black square in position $(T^*,1)$. The full triangular array of variables used in the construction of $\nu$ is enclosed by a red boundary.}
\label{fig:triangulararray}
\end{figure}

We observe that this algorithm has a complexity of $(T^*)^2$, as it is the number of steps needed to generate the variable $\bar{\fm}(0)$. It is worth noting that $-T^*$ can be constructed as the first hitting time of $0$ of the Markov chain $(X_n)$ with initial state
\[
X_0=\min\{j\geq1, w_j(0)\geq 1-\ell\}
\]
and with transition probabilities defined for all $j \geq 2$ and $i \geq j$ by
\[
    P(j,j-1) = 1 - (1-p)^{j-1} \text{ and } P(j,i) = p(1-p)^{i-1}
\]
where $p = \P(w_1(0) \geq 1-\ell)$, with
\[
  P(1,0) = \P(w_1(0) \geq \ell), \quad \P(1,1) = \P(1 - \ell \leq w_1(0) < \ell), \quad P(1,j) = p(1-p)^{j-1} \text{ for }j \geq 2.
\]
The quantity $X_n$ corresponds to the value of the variable $J$ at the end of the period when $t=-n$ in Algorithm \ref{algo:main}. In the example shown in Figure \ref{fig:triangulararray}, we have
\[
(X_0,X_{-1},X_{-2},X_{-3},X_{-4},X_{-5},X_{-6})=(4,3,2,1,2,1,0).
\]
Note that $T^*$ has exponential tails.

The choice of the parameter $\ell$ may have an important effect on the behaviour of the average complexity $\E((T^*)^2)$ of the algorithm. We plotted $\ell \mapsto \E((T^*)^2)$ in Figure \ref{fig:averageComplexity}, when the charge distribution is given by $F(\mathrm d x) = \ind{x \leq 1 } e^{x-1} \mathrm d x$. Additionally, as $p \to 0$, the quantity $\E((T^*)^2)$ grows to $\infty$. We estimated $\E((T^*)^2)$ for $F = p\delta_1+(1-p)\delta_{-\infty}$ and plotted this quantity as a function of $p$ in Figure \ref{fig:overComplexity}.

\begin{figure}[ht]
\centering
\includegraphics[width=6cm]{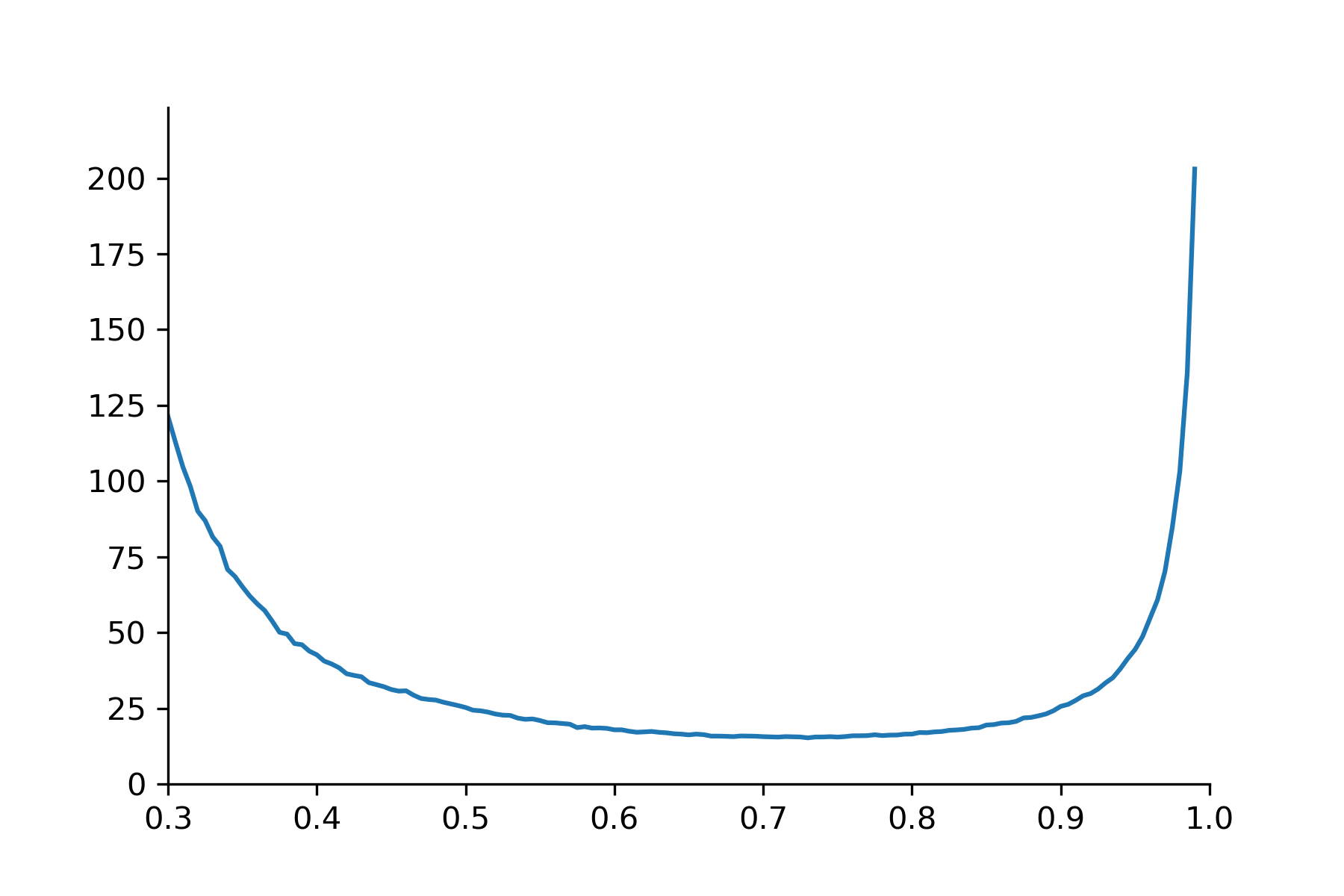}
\caption{Dependency in the parameter $\ell$ of the complexity of Algorithm \ref{algo:main} with a charge distribution $F(\mathrm d x) = \ind{x \leq 1 } e^{x-1} \mathrm d x$. The figure was obtained with a Monte Carlo simulation of $N=10^4$ copies of $T^*$ for $100$ different values of $\ell$. For this charge distribution, the  
Monte Carlo simulations give $C(F) = 0.4432\pm0.0006$.}
\label{fig:averageComplexity}
\end{figure}

We observe in Figure \ref{fig:averageComplexity} that different choices of the value $\ell$ can have a dramatic impact on the efficiency of Algorithm \ref{algo:main}. Choosing a value $\ell$ too small has the effect of making the first appearance of a triangular event too late. On the other hand, if $\ell$ is too big then with high probability, one will have $w_1(T) \leq \ell$, and thus the first ``successful'' triangular event will appear much later. For the distribution $F$ we chose, it appears that an optimal choice of $\ell$ seems to be around $\ell = 0.7$, which balances between these two extremes.

\begin{figure}[ht]
\centering
\includegraphics[width=6cm]{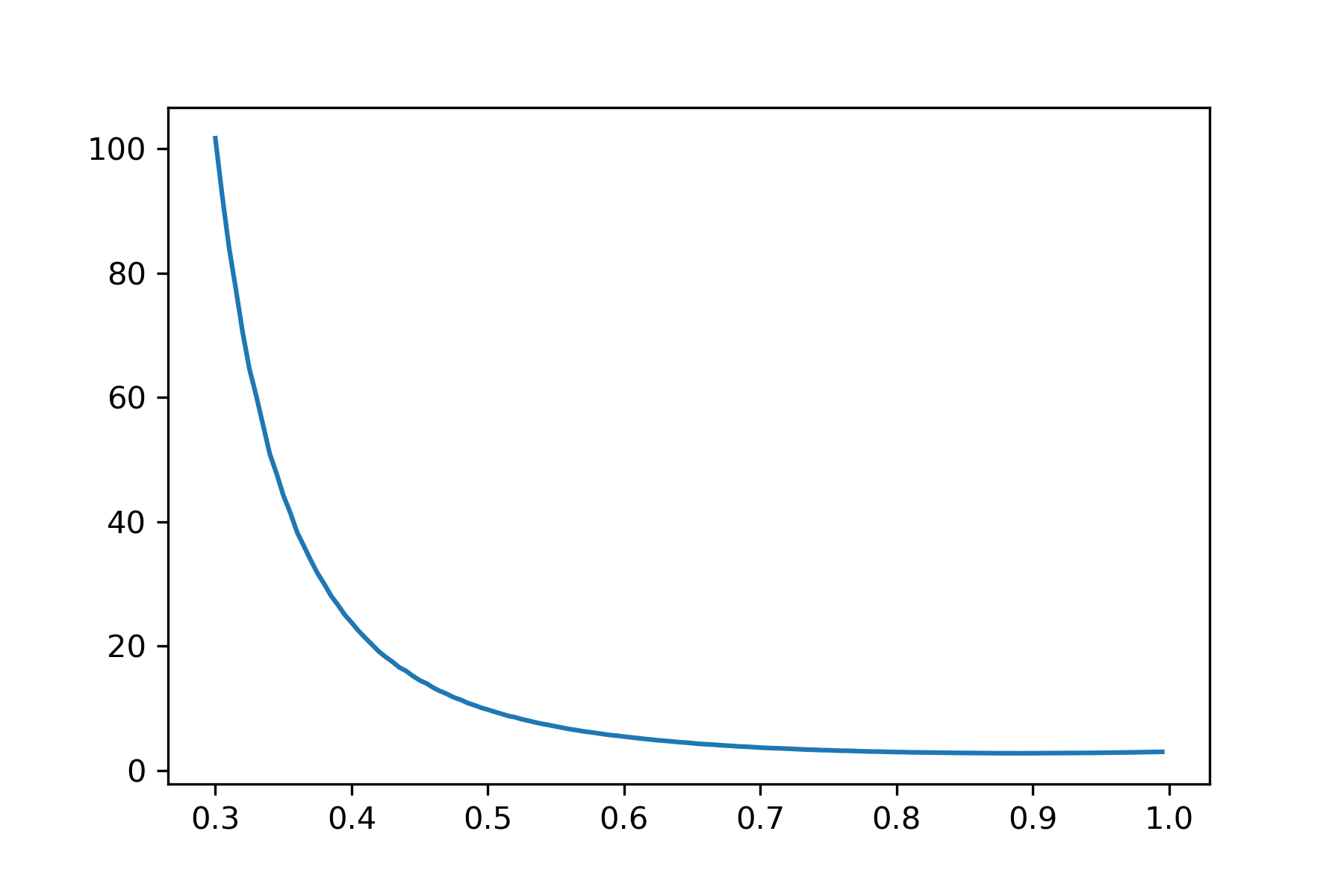}
\caption{Dependency in the parameter $p$ of the complexity of Algorithm \ref{algo:main} applied to the detection of the longest path in the Barak-Erd\H{o}s graph with parameter $p$. Figure obtained through Monte Carlo simulation of $N = 10^5$ copies of $T^*$ for 120 different values of $p$.}
\label{fig:overComplexity}
\end{figure}

We observe in Figure \ref{fig:overComplexity} that if $F$ puts a large mass on the negative half-line, the complexity of Algorithm \ref{algo:main} can become quite large. The function $p \mapsto \E(T^*)^2$ grows at least exponentially in $1/p$ as $p \to 0$ in the Barak-Erd\H{o}s graph, but we were not able to obtain a good estimate of this rate of increase.

\section{Further directions of research}
\label{sec:further}

In this article we considered last passage percolation on the directed complete graph, which has a total order on its vertex set. One extension of this would be to construct a perfect simulation algorithm for so-called directed slab graphs \cite{DFK12} where the set of vertices is only partially ordered. Another possible extension would be to add i.i.d.\ vertex weights with a distribution that has a finite essential supremum. In both cases, as well as in the setting considered in this paper, one should be able to obtain a perfect simulation algorithm if one replaces the i.i.d. weights by more general stochastic recursions with stationary drivers, as was considered in \cite{FK03}.

As discussed in the previous section, the complexity of our perfect simulation algorithm may dramatically vary with $\ell$. In the case of $F(\mathrm d x) = \ind{x \leq 1 } e^{x-1} \mathrm d x$ presented in Figure \ref{fig:averageComplexity}, there seems to be a unique optimal choice for $\hat\ell$ around $0.7$. It would be interesting to find some classes of distributions $F$ for which one has good bounds on the optimal value $\hat\ell$.

Yet another research direction would be the estimation of the constants appearing
for last passage percolation on a 2-dimensional version of the Barak-Erd\H{o}s directed graph
on the set $\N \times \N$ and whose edges are as follows:
if $u=(u_1,u_2), v=(v_1, v_2) \in \N \times \N$ are two vertices such that
$i_1 \le j_1$, $i_2 \le j_2$, then declare the pair $(u,v)$ as an edge 
directed from $u$ to $v$ with probability $p$, independently over
all such pairs. Then maximum length $L_n$ of all paths from $(1,1)$ to $(n, n^a)$, 
for a certain $a >0$, rescaled appropriately, converges weakly \cite{KT13} to 
a random variable having  a Tracy-Widom distribution depending on two parameters.
The estimation of these parameters is an open problem.

\section*{Acknowledgements}

We thank the referees for suggestions to improve the exposition.

SF was partially supported by the RFBR collaborative grant 19-51-15001 and TK, BM and SR were partially supported by the CNRS PRC collaborative grant CNRS-193-382 with the common title 
``Asymptotic and analytic properties of stochastic ordered graphs and infinite bin models''.

\newpage
\Addresses

\end{document}